\documentclass[a4paper]{amsart}

\usepackage{enumitem}
\usepackage{hyperref, aliascnt}

\newcommand{\andSep}{\,\,\,\text{ and }\,\,\,}
\newcommand{\CC}{{\mathbb{C}}}
\newcommand{\RR}{{\mathbb{R}}}
\newcommand{\ca}{$C^*$-algebra}
\newcommand{\spec}{\sigma}
\newcommand{\reg}{{\mathrm{reg}}}
\DeclareMathOperator{\dist}{dist}

\def\today{\number\day\space\ifcase\month\or   January\or February\or
   March\or April\or May\or June\or   July\or August\or September\or
   October\or November\or December\fi\   \number\year}

\newtheorem{lma}{Lemma}[section]

\newaliascnt{thmCt}{lma}
\newtheorem{thm}[thmCt]{Theorem}
\aliascntresetthe{thmCt}

\newaliascnt{corCt}{lma}
\newtheorem{cor}[corCt]{Corollary}
\aliascntresetthe{corCt}

\newaliascnt{prpCt}{lma}

\aliascntresetthe{prpCt}

\theoremstyle{definition}

\newaliascnt{dfnCt}{lma}

\aliascntresetthe{dfnCt}

\newcounter{theoremintro}

\newaliascnt{thmIntroCt}{theoremintro}
\newtheorem{thmIntro}[thmIntroCt]{Theorem}
\aliascntresetthe{thmIntroCt}

\title[Distance to regular elements and polar decompositions]{Distance to regular elements and polar decompositions in a C*-algebra}

\author{Hannes Thiel}
\address{Hannes~Thiel, 
Department of Mathematical Sciences, Chalmers University of Technology and University of
Gothenburg, Gothenburg SE-412 96, Sweden.}
\email{hannes.thiel@chalmers.se}
\urladdr{www.hannesthiel.org}

\thanks{
The author was partially supported by the Knut and Alice Wallenberg Foundation (KAW 2021.0140) and by the Swedish Research Council project grant 2024-04200.
}

\subjclass[2010]%
{Primary
46L05. % General theory of C*-algebras
Secondary
47C15, % Linear operators in C∗- or von Neumann algebras
47A63. %Linear operator inequalities 
}
\keywords{regular elements, well-supported elements, closed range operators, $C^*$-algebras, polar decomposition}
\date{\today}

\begin{document}

%==========================================================================================
\begin{abstract}
We show that the distance from an element of a \ca{} to the set of regular elements is the infimum of the $\delta>0$ for which the $\delta$-cut-down of the element admits a polar decomposition within the algebra.
This parallels results of Pedersen and Brown-Pedersen describing the distance to invertible and quasi-invertible elements through polar decompositions of cut-downs whose polar parts are unitaries or extreme partial isometries.
\end{abstract}

\maketitle

%==========================================================================================
%==========================================================================================
\section{Introduction}

%==========================================================================================
An element~$a$ of a \ca{} $A$ is called \emph{regular} if there exists $b \in A$ such that $a=aba$.
It is well known that this occurs precisely when~$a$ is \emph{well-supported}, meaning that~$0$ is isolated in the spectrum of~$a^*a$, or, equivalently, when~$a$ has \emph{closed range}, that is, when~$\pi(a)$ has closed range for every representation~$\pi$ of~$A$ on a Hilbert space;
see~\cite{HarMbe92GenInvCAlg1} and \cite[Paragraphs~II.3.2.9 and II.3.2.10]{Bla06OpAlgs}.
In this situation~$a$ admits a polar decomposition $a=v|a|$, where~$v$ is a partial isometry in~$A$ and $|a| = (a^*a)^{\frac{1}{2}}$. 
This mirrors the case of invertible elements, which always admit polar decompositions with unitary polar parts; 
the passage from invertible to regular elements corresponds to replacing the unitary by a partial isometry.

Every element of a von Neumann algebra admits a polar decomposition $a = v|a|$, and in that setting the partial isometry $v$ is uniquely determined by the requirement that~$v^*v$ and~$vv^*$ are the support projections of~$|a|$ and~$|a^*|$, respectively.
For an element $a$ of a \ca{}~$A$, the canonical polar decomposition $a=v|a|$ exists in the von Neumann algebra~$A^{**}$.
The positive part~$|a|$ always belongs to~$A$, and if~$a$ is regular then the partial isometry~$v$ lies in~$A$ as well.
In general, however, $v$ need not belong to~$A$.

In this paper we use a weaker notion: 
we say that $a\in A$ \emph{admits a polar decomposition in} $A$ if there exists some partial isometry $w \in A$ such that $a=w|a|$, without imposing the uniqueness conditions on~$w$ that are part of the definition of the polar decomposition in the von Neumann algebra setting.
This terminology is inspired by Pedersen \cite{Ped89ThreeQuavers}.
Even this weaker condition need not hold in general.  
For example, for the function $f\in C([0,1])$ defined by $f(t)=t\exp(i/t)$ for $t>0$ and $f(0)=0$, there is no partial isometry $w\in C([0,1])$ with $f=w|f|$.

\medskip

We establish a precise relationship between the distance from an element to the set $A_\reg$ of regular elements in a \ca{}~$A$ and the existence of polar decompositions for suitable cut-downs of the element.
This fits naturally alongside earlier results of Pedersen \cite{Ped87UnitaryExt} (see also R{\o}rdam's seminal paper \cite{Ror88AdvUnitaryRank}) and of Brown–Pedersen \cite{BroPed95GeomUnitBall}, who described the distances to the invertible and quasi-invertible elements in terms of polar decompositions whose partial isometries are, respectively, unitaries and extreme partial isometries.  
Recall that for $\delta>0$, the $\delta$-cut-down of a positive element~$a$ is defined by functional calculus for the function $t\mapsto\max\{t-\delta,0\}$, and is denoted $(a-\delta)_+$.
For a general element $a \in A$, we define its $\delta$-cut-down to be
\[
a_\delta := v (|a|-\delta)_+,
\]
where $a=v|a|$ is the canonical polar decomposition of $a$ in $A^{**}$.
We note that $vf(|a|)$ belongs to~$A$ for every positive, continuous function~$f$ vanishing at $0$;
in particular, each cut-down $a_\delta$ lies in~$A$.
The following is the main result of the paper.

%==========================================================================================
\begin{thmIntro}[\ref{prp:DistReg}]
\label{thmA}
Let~$a$ be an element in a \ca{} $A$, with polar decomposition $a=v|a|$ in~$A^{**}$.
Then for $\gamma \geq 0$, the following are equivalent:
\begin{enumerate}
\item
We have $\dist(a,A_\reg) \leq \gamma$.
\item
For every $\delta > \gamma$, there is a partial isometry $w \in A$ such that $wp_\delta = vp_\delta$, where $p_\delta$ denotes the spectral projection of~$|a|$ in~$A^{**}$ associated to $(\delta,\infty)$.
\item
For every continuous function $f\colon\RR\to\RR$ vanishing on a neighborhood of $[0,\gamma]$, the element $vf(|a|)$ admits a polar decomposition within~$A$.
\item
For every $\delta > \gamma$, the $\delta$-cut-down of $a$ admits a polar decomposition within~$A$.
\end{enumerate}
\end{thmIntro}

%==========================================================================================
The case $\dist(a,A_\reg)=0$ is of particular interest and has been considered before, for example in work of Xue, who proved in \cite[Proposition~2.1]{Xue07APDCAlg} that an element $a$ belongs to the closure of $A_\reg$ precisely when it has `approximate polar decompositions' in the sense that for every $\varepsilon>0$ there exists a partial isometry $w \in A$ such that $\| a - w|a| \| < \varepsilon$.
Our result strengthens this significantly: we obtain \emph{actual} polar decompositions for all cut-downs of $a$ and, more generally, for all elements of the form $v f(|a|)$, where $f$ is a positive, continuous function vanishing near $0$. 

In forthcoming work \cite{Thi27pre:RegRich}, we apply the results of this paper to study \ca{s} in which the regular elements are dense. 
In analogy with Brown--Pedersen's notion of extremal richness, we call such \ca{s} \emph{regularly rich}. 
Both real rank zero and stable rank one imply regular richness. 
We establish permanence of regular richness under suitable extensions, yielding natural examples that have neither real rank zero nor stable rank one, such as the Toeplitz algebra and the stable multiplier algebra of the Jiang--Su algebra.

Our broader aim is to use regular richness as a common framework for structural results that are known for both real rank zero and stable rank one. Prominent examples include $K_1$-injectivity \cite{Lin93ExpRank,Rie83DimSRKThy} and the solution of the Global Glimm Problem and the Rank Problem \cite{EllRor06Perturb,Thi20RksOps,AntPerRobThi22CuntzSR1}. 
We will investigate which of these results remain valid for regularly rich \ca{s}.

%==========================================================================================
%==========================================================================================
\section{Proof of the main result}

%==========================================================================================
An element $a$ in a \ca{} $A$ is said to be \emph{Moore-Penrose invertible} if there exists $b \in A$ such that $a=aba$ and $b=bab$, and such that~$ab$ and~$ba$ are projections.
In this case, the stated conditions determine~$b$ uniquely, and one calls it the \emph{Moore-Penrose inverse} of~$a$, denoted by~$a^\dag$.

It is known that every regular element in a \ca{} is Moore-Penrose invertible;
see \cite{HarMbe92GenInvCAlg1}.
Let us give some details.
Let~$a$ be a regular element in~$A$.
Then the partial isometry $v$ of the canonical polar decomposition $a=v|a|$ in $A^{**}$ belongs to $A$.
Further, there is some $\varepsilon>0$ such that the spectrum of $|a|$ satisfies
\[
\spec(|a|) \subseteq \{0\} \cup [\varepsilon,\infty).
\]
It follows that the function $g \colon \RR\to\RR$ given by $g(t)=0$ for $t \leq 0$ and $g(t)=\tfrac{1}{t}$ for $t>0$ is continuous on $\spec(|a|)$, and we have
\[
g(|a|)|a| = v^*v.
\]
The Moore-Penrose inverse of $a$ is then
\[
a^\dag = g(|a|)v^*.
\]

This shows that $a^\dag$ belongs to the \ca{} generated by $a$.
It follows that regularity is independent of the containing \ca{}.
In particular, an element is regular in~$A$ if and only if it is regular in~$A^{**}$.
Further, if~$a$ is contained in a corner~$pAp$, then~$a$ is regular in~$A$ if and only if~$a$ is regular in~$pAp$.
This will be used in the proofs below.

%==========================================================================================
\begin{lma}
\label{prp:MultRegInv}
Let $a$ be a regular element in a \ca{}~$A$.
Then $uav$ is regular for every invertible elements $u, v$ in the minimal unitization~$\widetilde{A}$.
\end{lma}
\begin{proof}
Let $b \in A$ satisfy $a=aba$. (For example, $b=a^\dag$.)
Set $c := v^{-1}bu^{-1} \in A$.
Then
\[
(uav)c(uav)
= (uav)\left(v^{-1}bu^{-1}\right)(uav)
= uabav
= uav,
\]
showing that $uav$ is regular.
\end{proof}

%==========================================================================================
Given an operator as a $2 \times 2$ block matrix with one off-diagonal entry vanishing, and one diagonal entry invertible, it is well-known that invertibility of the other diagonal entry characterizes invertibility of the whole operator;
see, for example, \cite[Lemma~2.3]{BroPed91CAlgRR0} or \cite[Lemma~3.4]{Rie83DimSRKThy}.
This was generalized to characterizations of left-invertibility and right-invertibility (and consequently of quasi-in\-ver\-ti\-bil\-i\-ty) by Brown and Pedersen in \cite[Proposition~2.1]{BroPed95GeomUnitBall}.
The next result is a version characterizing regularity.

%==========================================================================================
\begin{lma}
\label{prp:RegBlockMatrix}
Let $p$ and $q$ be projections in a unital \ca{}~$A$, and view elements in $A$ as $2 \times 2$ block matrices, with columns decomposed according to $p\oplus(1-p)$ and rows according to $q\oplus(1-q)$.
Let $x \in A$ with $qx(1-p)=0$, viewed as
\[
x=
\begin{bmatrix}
a & 0\\[3pt]
c & d
\end{bmatrix},
\qquad
\text{with }\quad
a=qxp, \quad
c=(1-q)xp, \quad 
d=(1-q)x(1-p).
\]

Assume that $a$ is invertible in $qAp$, that is, there exists $a^\dag \in pAq$ such that
\[
aa^\dag=q, \andSep
a^\dag a=p.
\]

Then $x$ is regular if and only if $d$ is regular in $(1-q)A(1-p)$ (equivalently, regular  in~$A$). 
\end{lma}
\begin{proof}
We have the factorization 
\[
x
=
\begin{bmatrix}
q & 0\\[2pt]
c\,a^\dag & 1-q
\end{bmatrix}
\begin{bmatrix}
a & 0\\[2pt]
0 & d
\end{bmatrix},
\]
and the left matrix is invertible.
It follows from \autoref{prp:MultRegInv} that $x$ is regular if and only if the diagonal matrix with entries $a$ and $d$ is regular.
This in turn is equivalent to regularity of $a$ and $d$.
Since~$a$ is regular by assumption, it follows that~$x$ is regular if and only if~$d$ is regular.
\end{proof}

%==========================================================================================
Let~$a$ be an element in a unital \ca{}~$A$, with polar decomposition $a=v|a|$ in~$A^{**}$, and let $p_\delta$ denote the spectral projection of~$|a|$ in~$A^{**}$ associated to $(\delta,\infty)$.
Thus, $p_\delta=1_{(\delta,\infty)}(|a|)$, where $1_{(\delta,\infty)}$ denotes the characteristic function of $(\delta,\infty)$, applied to $|a|$ using the Borel functional calculus in $A^{**}$.
Then $(vp_\delta)|a|_\delta$ is the polar decomposition of~$a_\delta$ in~$A^{**}$.

Given $\delta>0$, Pedersen \cite[Theorem~5]{Ped87UnitaryExt} showed that the partial isometry $vp_\delta$ admits an extension to a unitary in~$A$ whenever~$a$ is within distance $<\delta$ of an invertible element of~$A$;
see also \cite[Theorem~2.2]{Ror88AdvUnitaryRank}.
Similarly, if~$a$ is within distance $<\delta$ of a quasi-invertible element, then~$vp_\delta$ admits an extension to an extreme partial isometry in~$A$, by \cite[Theorem~2.2]{BroPed95GeomUnitBall}.
The following result is analogous, showing that $vp_\delta$ can be extended to a partial isometry in~$A$ whenever~$a$ has distance $<\delta$ to regular element of~$A$.

%==========================================================================================
\begin{lma}
\label{prp:ConstructPI}
Let~$a$ be an element in a \ca{}~$A$, and suppose that~$a$ lies within distance $<\delta$ of a regular element of~$A$.
Then there exists a partial isometry $w \in A$ such that
\[
w p_\delta 
= q_\delta w 
= v p_\delta 
= q_\delta v,
\]
where $a=v|a|$ is the polar decomposition of~$a$ in~$A^{**}$, and~$p_\delta$ and~$q_\delta$ denote the spectral projections of~$|a|$ and~$|a^*|$ in~$A^{**}$ associated to $(\delta,\infty)$.
\end{lma}
\begin{proof}
The proof proceeds along the same lines as the argument in \cite[Theorem~2.2]{BroPed95GeomUnitBall}.
We view~$A$ as an ideal in its minimal unitization~$\widetilde{A}$, and we use~$1$ to denote the unit of~$\widetilde{A}$.
If~$A$ is nonunital, then $\widetilde{A}/A \cong \CC$, and passing to biduals, we obtain a short exact sequence
\[
0 \to A^{**} \to \left( \widetilde{A} \right)^{**} \to \CC^{**}=\CC \to 0.
\]
Since~$A^{**}$ is unital, the bidual of $\widetilde{A}$ decomposes as a direct sum of $A^{**}$ and $\CC$.
For $\alpha>0$, we let~$p_\alpha$ and~$q_\alpha$ denote the spectral projections of~$|a|$ and~$|a^*|$ in~$A^{**}$ associated to $(\alpha,\infty)$.

Choose a regular element $x \in A$ such that $\|a-x\| < \delta$.
Set $\beta := \|a-x\|$ and choose $\gamma$ with
\[
\|a-x\| = \beta < \gamma < \delta.
\]
We may assume that $a \neq x$ (otherwise use $c:=x$ below) and set $y := \tfrac{1}{\beta}(x-a)$ so that
\[
x = a+\beta y, \andSep
\|y\| = 1.
\]

Define continuous functions $f,g \colon [0,\infty)\to\RR$ by:
\[
f(t) =
\begin{cases}
\gamma^{-1}, & t \leq \gamma,\\[2pt]
t^{-1}, & t \ge \gamma,
\end{cases}
\qquad
g(t) =
\begin{cases}
t\gamma^{-2}, & t \leq \gamma,\\[2pt]
t^{-1}, & t \ge \gamma.
\end{cases}
\]

Since $g$ vanishes at $0$, the element $g(|a|)v^*$ belongs to $A$.
Further, using that $\|\beta g\|_\infty = \beta \gamma^{-1} < 1$, it follows that $1+\beta g(|a|)v^*y$ is invertible in $\widetilde{A}$.
We can thus set
\[
b := x\bigl(1+\beta g(|a|)v^*y\bigr)^{-1}f(|a|). %\in A.
%b := (a + \beta w)\bigl(1+\beta g(|a|)v^*w\bigr)^{-1}f(|a|).
\]

Note that $b$ belongs to $A$ since $x \in A$.
Since $(1+\beta g(|a|)v^*y)^{-1}f(|a|)$ is invertible, it follows from \autoref{prp:MultRegInv} that $b$ is regular (in $\widetilde{A}$, and consequently also in $A$).

Working in the bidual of $\widetilde{A}$, we have
\[
q_\gamma x
= q_\gamma (a+\beta y)
= v p_\gamma |a| + \beta v p_\gamma v^* y
= v p_\gamma |a| \bigl( 1+\beta g(|a|)v^*y \bigr)
\]
and therefore
\[
q_\gamma b 
= v p_\gamma |a| f(|a|)
= v p_\gamma
= q_\gamma v.
\]

To simplify notation, we set
\begin{align*}
p_1 &:= p_\delta, &
p_2 &:= p_\gamma - p_\delta, &
p_3 &:= 1 - p_\gamma, \\
q_1 &:= q_\delta, &
q_2 &:= q_\gamma - q_\delta, &
q_3 &:= 1 - q_\gamma
\end{align*}
and then view elements in the bidual of $\widetilde{A}$ as $3 \times 3$ block matrices, with columns decomposed according to $p_1 \oplus p_2 \oplus p_3$ and rows according to $q_1 \oplus q_2 \oplus q_3$.

Since $(q_1+q_2)b=(q_1+q_2)v=v(p_1+p_2)$, it follows that~$b$ has the form
\[
b=
\begin{bmatrix}
vp_1 & 0 & 0\\[3pt]
0 & vp_2 & 0\\[3pt]
b_{31} & b_{32} & b_{33}
\end{bmatrix}
\qquad
\begin{aligned}
\text{with}\quad 
& b_{31}=q_3 b p_1,\quad 
  b_{32}=q_3 b p_2,\ \text{and}\\
& b_{33}=q_3 b p_3.
\end{aligned}
\]

Let $h_1,h_2\colon [0,\infty)\to[0,1]$ be some decreasing, continuous functions such that $h_1(t)=h_2(t)=1$ for $t \leq \gamma$ and $h_1(t)=h_2(t)=0$ for $t \geq \delta$, and such that $h_1h_2 = h_1$.
Now we set
\[
c := b - h_1(|a^*|)b\bigl( 1-h_2(|a|) \bigr) \in A.
\]

Since the spectral projections of~$|a^*|$ commute with~$h_1(|a^*|)$, and since~$h_1$ vanishes on $(\delta,\infty)$ and is constant of value~$1$ on $[0,\gamma]$, we have
\[
h_1(|a^*|)q_1 = 0, \quad
h_1(|a^*|)q_2 = q_2 h_1(|a^*|) q_2, \andSep
h_1(|a^*|)q_3 = q_3.
\]
Analogously, we have
\begin{align*}
&p_1\bigl( 1-h_2(|a|) \bigr) = p_1, \quad
p_2\bigl( 1-h_2(|a|) \bigr) = p_2\bigl( 1-h_2(|a|) \bigr)p_2, \andSep \\
&p_3\bigl( 1-h_2(|a|) \bigr)=0,
\end{align*}
and thus
\[
h_1(|a^*|)b\bigl( 1-h_2(|a|) \bigr)
=
\begin{bmatrix}
0 & 0 & 0\\[3pt]
0 & h_1(|a^*|)vp_2\bigl( 1-h_2(|a|) \bigr) & 0\\[3pt]
b_{31} & b_{32}\bigl( 1-h_2(|a|) \bigr) & 0
\end{bmatrix}.
\]
We further have
\[
h_1(|a^*|)vp_2\bigl( 1-h_2(|a|) \bigr)
= vp_2h_1(|a|)\bigl( 1-h_2(|a|) \bigr)
= 0
\]
and it follows that
\[
c=
\begin{bmatrix}
vp_1 & 0 & 0\\[3pt]
0 & vp_2 & 0\\[3pt]
0 & b_{32}h_2(|a|) & b_{33}
\end{bmatrix}.
\]

Next, we show that~$c$ is regular.
First, since the upper-left entry of~$c$ is~$vp_1$ and thus invertible in $q_1\big(\widetilde{A}\big)^{**}p_1$, it follows from \autoref{prp:RegBlockMatrix} that~$c$ is regular in~$A$ (equivalently, in~$\big(\widetilde{A}\big)^{**}$) if and only if 
\[
\begin{bmatrix}
vp_2 & 0\\[3pt]
b_{32}h_2(|a|) & b_{33}
\end{bmatrix}
\]
is regular in $(q_2+q_3)\big(\widetilde{A}\big)^{**}(p_2+p_3)$.
Applying the same argument again, we see that~$c$ is regular in~$A$ if and only if $b_{33}$ is regular in $q_3\big(\widetilde{A}\big)^{**}p_3$.
But an analogous argument applied to the regular element~$b$ shows that~$b_{33}$ is indeed regular, and hence so is~$c$.

Since~$c$ is regular, we have $c=w|c|$ for a partial isometry $w \in A$.
We have
\[
cp_1 = q_1c = vp_1 = q_1v
\]
from which it follows that 
\[
c^*cp_1
= c^*q_1v
= v^*q_1v
= p_1.
\]
We deduce that $r(c^*c)p_1=p_1$ for every polynomial $r$ with $r(1)=1$, and since $|c|=\lim_n r_n(c^*c)$ for such polynomials~$r_n$, we get $|c|p_1=p_1$.
Hence,
\[
wp_1 
= w|c|p_1
= cp_1
= vp_1.
\]

Analogously, we get $q_1cc^*=q_1$ and then $q_1|c^*|=q_1$, and so $q_1w=q_1v$.
This shows that the partial isometry $w$ satisfies the desired relations.
\end{proof}

%==========================================================================================
We are now ready to prove the main result (\autoref{thmA}), which characterizes the distance of an element to the set $A_\reg$ of regular elements in a \ca{} $A$.

%==========================================================================================
\begin{thm}
\label{prp:DistReg}
Let $a$ be an element in a \ca{} $A$, with polar decomposition $a=v|a|$ in $A^{**}$.
Then for $\gamma \geq 0$, the following are equivalent:
\begin{enumerate}
\item
We have $\dist(a,A_\reg) \leq \gamma$.
\item
For every $\delta > \gamma$, there is a partial isometry $w \in A$ such that $wp_\delta = vp_\delta$, where $p_\delta$ denotes the spectral projection of~$|a|$ in~$A^{**}$ associated to $(\delta,\infty)$.
\item
For every continuous function $f\colon\RR\to\RR$ vanishing on a neighborhood of $[0,\gamma]$, the element $vf(|a|)$ admits a polar decomposition within~$A$.
\item
For every $\delta > \gamma$, the $\delta$-cut-down of $a$ admits a polar decomposition within~$A$.
\end{enumerate}
\end{thm}
\begin{proof}
It follows directly from \autoref{prp:ConstructPI} that (1) implies~(2).
We verify the implications (2)$\Rightarrow$(3), (3)$\Rightarrow$(4), and (4)$\Rightarrow$(1).

`(2)$\Rightarrow$(3)':
Let $f\colon\RR\to\RR$ be a continuous function vanishing on $[0,\delta]$ for some $\delta>\gamma$.
By assumption, there is a partial isometry $w \in A$ such that $wp_\delta = vp_\delta$.
Then
\[
vf(|a|)
= v p_\delta f(|a|)
= w p_\delta f(|a|)
= w f(|a|),
\]
which shows that $w f(|a|)$ is the desired polar decomposition of $vf(|a|)$ inside $A$.

\medskip

`(3)$\Rightarrow$(4)':
This follows since for every $\delta>\gamma$, the $\delta$-cut-down of $a$ is of the form $vf(|a|)$ for a continuous function $f\colon\RR\to\RR$ vanishing on a neighborhood of $[0,\gamma]$.

\medskip

`(4)$\Rightarrow$(1)':
Let $\delta>\gamma$.
By assumption, $a_\delta$ admits a polar decomposition $a_\delta = w (|a|-\delta)_+$ in $A$.
We view $A$ as an ideal in its minimal unitization $\widetilde{A}$, and we use $1$ to denote the unit of $\widetilde{A}$.
Given $\varepsilon>0$, the element $\varepsilon 1 + (|a|-\delta)_+$ is strictly positive and therefore invertible in $\widetilde{A}$.
It follows from \autoref{prp:MultRegInv} that the element
\[
w\bigl( \varepsilon 1 + (|a|-\delta)_+ \bigr)
\]
is regular.
Note that $w\bigl( \varepsilon 1 + (|a|-\delta)_+ \bigr)$ belongs to~$A$, and therefore to~$A_\reg$.
We deduce that
\begin{align*}
\dist(a,A_\reg)
&\leq \left\| a - w\bigl( \varepsilon 1 + (|a|-\delta)_+ \bigr) \right\| \\
&\leq \| a - a_\delta \| + \left\| w (|a|-\delta)_+ - w\bigl( \varepsilon 1 + (|a|-\delta)_+ \bigr) \right\| 
\leq \delta + \varepsilon.
\end{align*}
Since this holds for all $\varepsilon>0$ and all $\delta>\gamma$, we get $\dist(a,A_\reg) \leq \gamma$.
\end{proof}

%==========================================================================================
\begin{cor}
\label{prp:DistZero}
Let $a$ be an element in a \ca{} $A$, with polar decomposition $a=v|a|$ in $A^{**}$.
Assume that $a$ belongs to the closure of~$A_\reg$.
Then:
\begin{enumerate}
\item
For every $\delta > 0$, there is a partial isometry $w \in A$ such that $wp_\delta = vp_\delta$, where~$p_\delta$ denotes the spectral projections of $|a|$ in~$A^{**}$ associated to $(\delta,\infty)$.
\item
For every continuous function $f\colon\RR\to\RR$ vanishing on a neighborhood of~$\{0\}$, the element $vf(|a|)$ admits a polar decomposition within~$A$.
\item
For every $\delta > 0$, the $\delta$-cut-down of~$a$ admits a polar decomposition within~$A$.
\end{enumerate}
\end{cor}

%==========================================================================================
%==========================================================================================
%\bibliographystyle{../../aomalphaMyShort}
%\bibliography{../../References}

\providecommand{\bysame}{\leavevmode\hbox to3em{\hrulefill}\thinspace}
\providecommand{\noopsort}[1]{}
\providecommand{\mr}[1]{\href{http://www.ams.org/mathscinet-getitem?mr=#1}{MR~#1}}
\providecommand{\zbl}[1]{\href{http://www.zentralblatt-math.org/zmath/en/search/?q=an:#1}{Zbl~#1}}
\providecommand{\jfm}[1]{\href{http://www.emis.de/cgi-bin/JFM-item?#1}{JFM~#1}}
\providecommand{\arxiv}[1]{\href{http://www.arxiv.org/abs/#1}{arXiv~#1}}
\providecommand{\doi}[1]{\url{http://dx.doi.org/#1}}
\providecommand{\MR}{\relax\ifhmode\unskip\space\fi MR }
% \MRhref is called by the amsart/book/proc definition of \MR.
\providecommand{\MRhref}[2]{%
  \href{http://www.ams.org/mathscinet-getitem?mr=#1}{#2}
}
\providecommand{\href}[2]{#2}

\end{document}